\documentclass[twoside,11pt]{article}
\usepackage{float} 
\usepackage{graphicx, subfigure}
\usepackage[top=1in,bottom=1in,left=1in,right=1in]{geometry}
\usepackage[sort&compress]{natbib} 
\usepackage{amsfonts, amsmath, amssymb, amsthm, constants, bbm}
\usepackage{MnSymbol}
\usepackage{amsthm} 
\usepackage{enumitem}
\usepackage{xcolor}
\usepackage{algorithm}
\usepackage{algorithmic}
\usepackage{amsmath}

\usepackage{booktabs}
\usepackage{color}
\usepackage{xcolor}
\definecolor{darkred}{RGB}{100,0,0}
\definecolor{darkgreen}{RGB}{0,100,0}
\definecolor{darkblue}{RGB}{0,0,150}

\usepackage{hyperref}
\hypersetup{colorlinks=true, linkcolor=darkred, citecolor=darkgreen, urlcolor=darkblue}
\usepackage{url}

\def\d{{\rm d}}

\newtheorem{theorem}{Theorem}[section]

\newtheorem{lem}[theorem]{Lemma}

\theoremstyle{remark}

\newtheorem{asp}{Assumption}

\def\beq{\begin{equation}} 
\def\eeq{\end{equation}}
\def\beqn{\begin{eqnarray*}}
\def\eeqn{\end{eqnarray*}}
\def\Bitem{\begin{itemize}\setlength{\itemsep}{.2in}}
\def\bitem{\begin{itemize}\setlength{\itemsep}{.05in}}
\def\eitem{\end{itemize}}
\def\Benum{\begin{enumerate}\setlength{\itemsep}{.2in}}
\def\benum{\begin{enumerate}\setlength{\itemsep}{.05in}}
\def\eenum{\end{enumerate}}
\def\bmult{\begin{multline*}}
\def\emult{\end{multline*}}
\def\bcenter{\begin{center}}
\def\ecenter{\end{center}}
\def\bframe{\begin{frame}}
\def\eframe{\end{frame}}

\newcommand{\thmref}[1]{Theorem~\ref{thm:#1}}

\newcommand{\lemref}[1]{Lemma~\ref{lem:#1}}
\newcommand{\secref}[1]{Section~\ref{sec:#1}}
\newcommand{\figref}[1]{Figure~\ref{fig:#1}}

\newcommand{\algref}[1]{Algorithm~\ref{alg:#1}}
\newcommand{\aspref}[1]{Assumption~\ref{asp:#1}}

\DeclareMathOperator*{\argmin}{arg\, min}

\def\cA{\mathcal{A}}

\def\cG{\mathcal{G}}
\def\cH{\mathcal{H}}

\def\cM{\mathcal{M}}

\def\cR{\mathcal{R}}

\def\cT{\mathcal{T}}
\def\cU{\mathcal{U}}
\def\cV{\mathcal{V}}
\def\cW{\mathcal{W}}

\newcommand{\E}{\operatorname{\mathbb{E}}}
\renewcommand{\P}{\operatorname{\mathbb{P}}}

\def\1{\mathbbm{1}}

\definecolor{purple}{rgb}{0.4,.1,.9}

\newcommand\blfootnote[1]{%
  \begingroup
  \renewcommand\thefootnote{}\footnote{#1}%
  \addtocounter{footnote}{-1}%
  \endgroup
}

\begin{document}
\thispagestyle{empty}

\title{Some Theory for Texture Segmentation}
\author{Lin Zheng}
\date{}
\maketitle

\blfootnote{The author is with the Department of Mathematics, University of California, San Diego, USA.  Contact information is available  \href{http://www.math.ucsd.edu/people/graduate-students/}{here}.}

\begin{abstract}
In the context of texture segmentation in images, and provide some theoretical guarantees for the prototypical approach which consists in extracting local features in the neighborhood of a pixel and then applying a clustering algorithm for grouping the pixel according to these features. 
On the one hand, for stationary textures, which we model with Gaussian Markov random fields, we construct the feature for each pixel by calculating the sample covariance matrix of its neighborhood patch and cluster the pixels by an application of k-means to group the covariance matrices. We show that this generic method is consistent. 
On the other hand, for non-stationary fields, we include the location of the pixel as an additional feature and apply single-linkage clustering. We again show that this generic and emblematic method is consistent. 
We complement our theory with some numerical experiments performed on both generated and natural textures.
\end{abstract}

\section{Introduction} \label{sec:intro}

Texture segmentation fits within the larger area of image segmentation, with a particular focus on images that contain textures. The goal, then, is to partition the image, i.e., group the pixels, into differently textured regions. 
Texture segmentation, and image segmentation more generally, is an important task in computer vision and pattern recognition, being widely applied to areas such as scene understanding, remote sensing and autonomous driving \citep{pal1993review, zhang2006advances, reed1993review, liu2019bow}.

At least in recent decades, texture segmentation methods are almost invariably based on extracting local features around each pixel, such as SIFT \citep{lowe1999object}, which are then fed into a clustering algorithm, such as k-means. 
An emblematic approach in this context is that of \cite{shi2000normalized}, who used the pixel value as feature, arguably the simplest possible choice, and applied a form of spectral clustering to group the pixels. 
The process is similar to what is done in the adjacent area of texture classification, the main difference being that a classification method is used instead of a clustering algorithm \citep{varma2004statistical, randen1999filtering}.

Although this basic approach has remained essentially unchanged, the process of extracting features has undergone some important changes over the years, ranging from the use of sophisticated systems from applied harmonic analysis such as Gabor filters or wavelets \citep{dunn1995optimal, grigorescu2002comparison, jain1990unsupervised, unser1995texture, weldon1996design, randen1999texture} to multi-resolution or multiscale aggregation approaches \citep{galun2003texture, mao1992texture}, among others \citep{malik2001contour, hofmann1998unsupervised}, to the use deep learning, particularly in the form convolutional neural networks (CNN), whose success is attributed to the capability of CNN to learn a hierarchical representation of raw input data \citep{long2015fully, ronneberger2015u, milletari2016v, badrinarayanan2017segnet}.
See \citep{humeau2019texture} for a recent survey.

While the vast majority of the work in texture segmentation, as in image processing at large, is applied, we contribute some theory by establishing the consistency of the basic approach described above. We do so in a stylized setting which is nonetheless a reasonable mathematical model for the problem of texture segmentation.
Markov random fields (MRF) are common models for textures \citep{cross1983markov, geman1986markov}, and arguably the most popular in theoretical texture analysis \citep{rue2005gaussian, arias2018detecting, verzelen2010adaptive, verzelen2010high, verzelen2009tests}. This is the model that we use.
Although supplanted by the more recent feature extraction methods mentioned above, which in recent years are invariably nonparametric, Gaussian MRF in particular remain the most commonly-used parametric model for textures, also used in the development of methodology not too long ago \citep{chellappa1985classification, zhu1998filters, manjunath1991unsupervised, paciorek2006spatial}. 
When textures are modeled by stationary Gaussian MRF, what characterizes them is the covariance structure, so that in congruence with adopting Gaussian MRF as models for textures, when assumed stationary the feature we extract is the (local) covariance.
When textures are not assumed stationary, we also incorporate location as an additional feature, as the covariance structure may change within a textured region. 

The basic approach calls for applying a clustering algorithm to the extracted features. Features are typically represented by (possibly high-dimensional) feature vectors, as is the case with the features that we work with, and thus a large number of clustering methods are applicable, some of them coming with theoretical guaranties such as k-means \citep{arthur2006k}, Gaussian mixture models \citep{dasgupta1999learning, vempala2004spectral, hsu2013learning}, hierarchical clustering \citep{dasgupta2005performance, dasgupta2010hierarchical}, including single-linkage clustering \citep{arias2011clustering}, and spectral clustering \citep{ng2002spectral}. 
In this paper, we use k-means in the context of stationary textures and singe-linkage clustering in the context of non-stationary textures. 

The paper is organized as follows. 
In \secref{stationary}, we consider stationary textures, which is done by the extraction of local second moment information on patches and the application of k-means. 
In \secref{non-stationary}, we consider non-stationary textures, where we also include location as a feature and we apply instead single-linkage clustering. 
In \secref{numerics}, we present the result of some numerical experiments, mostly there to illustrate the theory developed in the main part of the paper. Both synthetic and natural textures are considered.

\section{Stationary Textures}
\label{sec:stationary}

In this section we consider textures to be stationary. The model we adopt and the method we implement are introduced in \secref{stationary_model} and \secref{stationary_method}. We then establish in \secref{stationary_theory} the consistency of a simple incarnation of the basic approach.

\subsection{Model}
\label{sec:stationary_model}

We have a pixel image $X$ of size $n \times n$, that we assume is partitioned into two sub-regions $\cG_0$ and $\cG_1$ by curve $\bar{\cG}$. $\cG_0$ is a stationary Gaussian Markov random field with mean $0$ and autocovariance matrix $A_0$. $\cG_1$ is a stationary Gaussian Markov random field with mean $0$ and autocovariance matrix $A_1$. In image $X$, we pick up $n^2$ pixels with equal intervals, and get observations
\begin{equation}
\{X_{t}\},\;\;\;\;\; t \in  \cT := \{1,2,\cdots, n\}^2.
\end{equation}
To estimate curve $\bar{\cG}$, we need to cluster the $n^2$ pixels into two groups.

\subsection{Methods}
\label{sec:stationary_method}
We define scanning patches as follows. To simplify the presentation assume $n$ is the square of an integer (namely $n = m^2$ for some integer m). For $\forall t = (t_1,t_2) \in \cT$, pick up patch $S_t$ with size $(2m+1)\times (2m+1)$,
\begin{equation}
\label{patch}
S_t=
\begin{pmatrix}
   X_{t+(-m,-m)} & X_{t+(-m, -m+1)} & \cdots & X_{t+(-m,m)} \\
   X_{t+(-m+1,-m)} & X_{t+(-m+1,-m+1)} & \cdots & X_{t+(-m+1,m)} \\
   \vdots & \vdots & \cdots& \vdots \\
   X_{t+(m,-m)} & X_{t+(m,-m+1)} & \cdots& X_{t+(m,m)}
\end{pmatrix},
\end{equation}
Next, autocovariance is defined based on scanning patches. For $\forall t = (t_1, t_2) \in \cT$ and $ \forall i=(i_1,i_2) \in \cM :=\{-m, -m+1, \cdots, m-1, m\}^2$, define true autocovariance and sample autocovariance as follows

\begin{equation}
C_t(i) = \text{Mean of }\{ \E(X_t \cdot X_{t+i})\;|\; \text{both} \;X_{t}\; \text{and} \;X_{t+i}\; \text{are in} \;S_t\},
\end{equation}
\begin{equation} 
\hat{C}_t(i) = \text{Mean of }\{ X_t \cdot X_{t+i}\;|\; \text{both} \;X_{t}\;\text{and} \;X_{t+i}\; \text{are in} \;S_t\}.
\end{equation}
Denote the vectorizations of $\{C_t(i)\}_{i \in \cM}$ and $\{\hat{C}_t(i)\}_{i \in \cM}$ to be $C_t$ and $\hat{C}_t$ respectively. Here $C_t$ is the true feature of pixel $X_t$ and $\hat{C}_t$ is the observed feature of pixel $X_t$. 

Also based on scanning patches, we define following three sets
\begin{equation}
\cH_0 = \{ t \in  \cT \;|\; S_t \subset \cG_0 \},
\end{equation}
\begin{equation}
\cH_1 = \{ t \in  \cT \;|\; S_t \subset \cG_1 \},
\end{equation}
\begin{equation}
\cH = \{ t \in  \cT \;|\; S_t \cap  \cG_0 \neq \emptyset \;\text{and} \; S_t \cap  \cG_1 \neq \emptyset \}.
\end{equation}
Here $\cG_0$ and $\cG_1$ are both stationary fields, so all elements in set $\{C_t\}_{\forall t \in \cH_0}$ are the same and we denote it as $C^0$. Similarly, all elements in set $\{C_t\}_{\forall t \in \cH_1}$ are the same and we denote it as $C^1$. Define template autocovariance $E = 
 \begin{pmatrix}
   C^0 \\
   C^1
  \end{pmatrix}$.
  
Then we introduce membership matrix.
Define $n^2 \times 2$ true membership matrix $W$ such that for $\forall t = (t_1, t_2) \in \cT$,
\begin{equation}
W_t = \text{the} \; n(t_1-1)+t_2\; th \;\text{row of matrix} \; W =
\left\{
\begin{array}{lr}
    (1,0), &  \text{if} \;\;t=(t_1,t_2) \in \cG_0,\\
   (0,1),  & \text{if}\;\; t=(t_1,t_2) \in \cG_1.
\end{array}
\right.
\end{equation}
Also define the set of membership matrices $\cW_{n,2}$ as follows
 \begin{equation}
\cW_{n,2} = \{n^2 \times 2 \;\text{matrices with rows}\; (0,1) \;\text{or}\; (1,0)\}.
\end{equation}
Based on above calculations and definitions, we define k-means clustering estimation as
\begin{equation}
(\hat{W}, \hat{E}) = \argmin_{W \in \cW_{n,2}, E\in \mathbb{R}_{2\times (2m+1)^2}}\sum_{t \in \cT}\|(WE)_{t}-\hat{C}_{t}\|^2_{\infty},
\end{equation}
where $(WE)_t$ is the $n(t_1-1)+t_2\; th$ row of matrix $WE$. 

In practice k-means can not be solved exactly, however, there exists polynomial time algorithm which obtains approximation ($\hat{W}, \hat{E}$) satisfying following equation \eqref{11}, such as $(1 +\epsilon)$-approximate method in \citep{kumar2004simple}.
\begin{equation}
\sum_{t \in \cT}\|(\hat{W}\hat{E})_{t}-\hat{C}_{t}\|^2_{\infty} \leq (1+\epsilon)\cdot \min_{W \in \cW_{n,2}, E\in \mathbb{R}_{2\times (2m+1)^2}}\; \sum_{t \in \cT}\|(WE)_{t}-\hat{C}_{t}\|^2_{\infty},\label{11}
\end{equation}
where $\hat{W} \in \cW_{n,2}$ and $\hat{E}\in \mathbb{R}_{2\times (2m+1)^2}$. Thus we cluster the $n^2$ pixels into two groups by membership matrix estimation $\hat{W}$. As a summary, we provide the procedure of
k-means algorithm in \algref{alg1}.
\begin{algorithm}[H]
  \caption{Texture Segmentation with K-means Algorithm}
  \label{alg:alg1}
  \begin{algorithmic}[1]
    \REQUIRE
      Observations $\{X_t\}_{t\in \cT}$, approximation error $\epsilon$.
    \ENSURE
      Membership matrix estimation $\hat{W}$.
    \STATE 
    For $\forall t = (t_1,t_2) \in \cT$, pick up patch $S_t$.
    \STATE
    Calculate sample autocovariance $\{\hat{C}_t(i)\}_{i \in \cM}$ and obtain observed features $\{\hat{C}_t\}_{t\in \cT}$.
    \STATE
    Define template autocovariance $E$ and the set of membership matrices $\cW_{n,2}$.
    \STATE
    Obtain k-means approximation solution ($\hat{W}, \hat{E}$) which satisfies \eqref{11}.
    \RETURN $\hat{W}$.
  \end{algorithmic}
\end{algorithm}

Define the set of $2\times 2$ permutation matrices
\begin{equation}
\Phi_2=\left\{\begin{pmatrix}
   1 & 0 \\
   0 & 1
  \end{pmatrix},\begin{pmatrix}
   0 & 1 \\
   1 & 0
  \end{pmatrix}\right\},
\end{equation}
then calculate
\begin{equation}
\hat{Q} = \argmin_{Q\in \Phi_2} \;\sum_{t\in \cT}\|(\hat{W}Q)_t-W_t\|^2_\infty.
\end{equation}
Next, define the set of mistakenly clustered elements to be $\cR$ as follows
\begin{equation}
\cR = \{t\in \cT: (\hat{W}\hat{Q})_t \neq W_t\},
\end{equation}
then clustering error rate is
\begin{equation}
|\cR|/n^2 = \frac{1}{n^2} \sum_{t\in \cT}\|(\hat{W}\hat{Q})_t-W_t\|^2_\infty.
\end{equation}

\subsection{Theory}
\label{sec:stationary_theory}
Firstly we introduce following assumptions.

\begin{asp}
\label{asp:stationary}
Both $\cG_0$ and $\cG_1$ are wide-sence stationary Gaussian Markov random fields.
\end{asp}

\begin{asp}
\label{asp:difference}
Let $\Delta = \|C^0-C^1\|_{\infty}$. For $\forall \beta >1$,
\begin{equation}
\frac{(\log n)^\beta}{\Delta^2 n}\rightarrow 0 \;\;\;\;\text{as}  \;\;\;\;n\rightarrow \infty.
\end{equation}
\end{asp}

\begin{asp}
\label{asp:curve} Define 
$C_{t0}=C^0$ for $\forall t \in \cG_0$ and $C_{t0}=C^1$ for $\forall t \in \cG_1$. With the same $\beta$ in \aspref{difference},
\begin{equation}
\sum_{t\in \cH}\|C_t-C_{t0}\|^2_{\infty}\leq \frac{n (\log n)^\beta}{24}.
\end{equation}
\end{asp}

Next, before introducing the theory, we indicate the error bound of $\|\hat{C}_t- C_t\|_{\infty}$.

\begin{lem}
\label{lem:consistency}
Under \aspref{stationary}, for $\forall t \in \cT$ and $\forall a >0$, there exists a constant $J$ such that
\begin{equation} \P(\|\hat{C}_t- C_t\|_{\infty}>a) \leq 2(2\sqrt{n}+1)^2 \exp(-J\cdot a^2n).
\end{equation}
\end{lem}

\begin{proof}
First let $S_t^v$ be the vectorization of $S_t$, then $S_t^v$ is a vector of length $(2m+1)^2$
\begin{equation}
S_t^v=
 \begin{pmatrix}
   X_{t+(-m,-m)} \\
   \vdots \\
   X_{t+(-m,m)} \\
   X_{t+(-m+1,-m)} \\
   \vdots\\
   X_{t+(-m+1,m)} \\
   \vdots \\
   \vdots \\
   \vdots \\
    X_{t+(m,-m)} \\ 
    \vdots\\
    X_{t+(m,m)}
  \end{pmatrix}.
\end{equation}
Then for $\forall i=(i_1,i_2) \in \cM$, there exists a matrix $A_i$ such that
\begin{equation}
\hat{C}_t(i) = \frac{1}{(2m+1-|i_1|)(2m+1-|i_2|)} (S_t^v)^TA_iS_t^v,
\end{equation}
where $A_i$ is a $(2m+1)^2\times (2m+1)^2$ matrix with elements $0$ and $1$, and the number of 1 is less than $(2m+1-|i_1|)(2m+1-|i_2|)$.

Since the field is stationary, suppose $S_t^v \sim N(0, \Sigma)$, where $\Sigma$ is non-negative. Let $\Sigma = U\Lambda U^T$ be the spectral decomposition of $\Sigma$. Define 
\begin{equation}
Y_t = U^TS_t^v,
\end{equation}
then 
\begin{equation}
Y_t \sim N(0, U^T\Sigma U) = N(0, \Lambda).
\end{equation}
So 
\begin{align}
\hat{C}_t(i) &=  \frac{1}{(2m+1-|i_1|)(2m+1-|i_2|)} (S_t^v)^TA_iS_t^v \\
&=  \frac{1}{(2m+1-|i_1|)(2m+1-|i_2|)} (UY_t)^TA_i(UY_t)\\
&= 
\frac{1}{(2m+1-|i_1|)(2m+1-|i_2|)} Y_t^T(U^TA_iU)Y_t.
\end{align}
By Hanson-Wright inequality in
\citet*{rudelson2013hanson}, for $ \forall a>0$, there exist constants $K_1$ and $J_1$, such that
\begin{align}
&\P(|\hat{C}_t(i)- C_t(i)]|>a)\\
&= \P(|Y_t^T(U^TA_i U)Y_t- \E[Y_t^T(U^TA_i U)Y_t]|>a\cdot (2m+1-|i_1|)(2m+1-|i_2|))\\
&\leq 2 \exp\Big(-J_1\cdot \min\Big\{\frac{a^2(2m+1-|i_1|)^2(2m+1-|i_2|)^2}{K_1^4\|U^TA_i U\|^2_F}, \frac{a(2m+1-|i_1|)(2m+1-|i_2|)}{K_1^2\|U^TA_i U\|_2}\Big\}\Big).
\end{align}
Next we focus on $\|U^T A_i U\|_F$ and $\|U^T A_i U\|_2$. Since $U$ is an orthogonal matrix, 
\begin{align}
\|U^TA_i U\|_2 =  \|A_i\|_2 \leq \|A_i\|_1 
= \max_k{\sum}_{l=1}^{(2m+1)^2}|A_i(k,l)| = 1
\end{align}
and
\begin{align}
\|U^T A_i U\|_F &=  \sqrt{\text{Sum of eigenvalues of } (U^T A_i U)^T(U^T A_i U)}\\
&\leq \sqrt{(2m+1)^2 \cdot \lambda_{max}((U^TA_i U)^T(U^TA_i U))}\\
&= \sqrt{(2m+1)^2} \cdot \|U^TA_i U\|_2\\
&= 2m+1.
\end{align}
Then for $\forall i=(i_1,i_2) \in \cM$, there exist constants $K_1$, $J_1$ and $J$ such that
\begin{align}
&\P(|\hat{C}_t(i)- C_t(i)]|>a) \\
&\leq 2 \exp\Big(-J_1\cdot \min\Big\{\frac{a^2(2m+1-|i_1|)^2(2m+1-|i_2|)^2}{K_1^4(2m+1)^2}, \frac{a(2m+1-|i_1|)(2m+1-|i_2|)}{K_1^2}\Big\}\Big)\\
&\leq 2 \exp(-J\cdot \min\{a^2m^2, am^2\}).
\end{align}
So when $a$ is small enough, we have
\begin{equation}
\P(|\hat{C}_t(i)- C_t(i)]|>a) 
\leq 2 \exp(-J\cdot a^2m^2).
\end{equation}
Next since
\begin{equation}
\|\hat{C}_t-C_t\|_{\infty} = \max_{i \in \cM}|\hat{C}_t(i)- C_t(i)|,
\end{equation}
by Union bound, for $\forall a>0$,
\begin{align}
\P(\|\hat{C}_t- C_t\|_{\infty}>a) &\leq 2(2m+1)^2 \exp(-J\cdot a^2m^2)\\
&= 2(2\sqrt{n}+1)^2 \exp(-J\cdot a^2n).
\end{align}
\end{proof}

\begin{lem}
\label{lem:cluster}
Let $\Delta = \|C^0-C^1\|_{\infty}$. Define $\cA_k = \{t \in \cG_k : \|(\hat{W}\hat{E})_t-C^k\|_{\infty} \geq \Delta/2 \}, k= 0,1$, and $\cA' = \cA_0 \cup \cA_1$, we have $\cT \backslash \cA'=(\cG_0 \backslash \cA_0)\cup (\cG_1 \backslash \cA_1) $. Then all the elements in $\cT \backslash \cA'$ are clustered correctly.
\end{lem}
\begin{proof}
On the one hand, for $\forall t\in \cG_0 \backslash \cA_0$ and $\forall s \in \cG_1 \backslash \cA_1$, by contradiction, if $(\hat{W}\hat{E})_t= (\hat{W}\hat{E})_s$,
\begin{align}
\Delta = \|C^0-C^1\|_{\infty} &\leq  \|C^0-(\hat{W}\hat{E})_t\|_{\infty}+ \|(\hat{W}\hat{E})_t- (\hat{W}\hat{E})_s\|_{\infty}+\|(\hat{W}\hat{E})_s-C^1\|_{\infty}\\
&< \Delta/2+0+\Delta/2\\
&= \Delta,
\end{align}
which is conflicted by itself, so $(\hat{W}\hat{E})_t \neq (\hat{W}\hat{E})_s$. On the other hand, suppose $t,s \in \cG_0 \backslash \cA_0$ or $t,s \in \cG_1 \backslash \cA_1$,  by contradiction, if $(\hat{W}\hat{E})_t \neq (\hat{W}\hat{E})_s$, $\hat{W}\hat{E}$ has at least three distinct rows, however, according to the structure of $\hat{W}\hat{E}$, it has exactly two distinct rows, which is a conflict. So $(\hat{W}\hat{E})_t=  (\hat{W}\hat{E})_s$.
Thus, all the elements in $\cT \backslash \cA'$ are clustered correctly.
\end{proof}

Next, we introduce the theory for k-means clustering algorithm.
\begin{theorem}
Under Assumption~\ref{asp:stationary}, \ref{asp:difference} and \ref{asp:curve}, consider k-means clustering in \algref{alg1}. For $\forall \beta>1$, there exists a constant $J$, as $n \rightarrow \infty$, 
\begin{align}
\P\Big(|\cR|/n^2 >\frac{(\log n)^{\beta}}{\Delta^2 n}\Big)
&\leq 2(2\sqrt{n}+1)^2n^2 \exp(-J\cdot (\log n)^{\beta}) \rightarrow 0,
\end{align}
where $|\cR|/n^2 $ is the clustering error rate. Here $\frac{(\log n)^{\beta}}{\Delta^2 n} \rightarrow 0$ as $n \rightarrow \infty$.
\end{theorem}

\begin{proof}
By \algref{alg1} in \secref{stationary_method}, we have
\begin{equation}
\sum_{t \in \cT}\|(\hat{W}\hat{E})_{t}-\hat{C}_{t}\|^2_{\infty} \leq (1+\epsilon)\cdot \min_{W \in \cW_{n,2}, E\in \mathbb{R}_{2\times (2m+1)^2}}\; \sum_{t \in \cT}\|(WE)_{t}-\hat{C}_{t}\|^2_{\infty},\label{47}
\end{equation}
where $\hat{W} \in \cW_{n,2}$, $\hat{E}\in \mathbb{R}_{2\times (2m+1)^2}$. Without loss of generality, set $\epsilon<1$, then
\begin{align}
\sum_{t \in \cT}\|(\hat{W}\hat{E})_{t}-\hat{C}_{t}\|^2_{\infty} &\leq (1+\epsilon) \sum_{t \in \cT}\|C_{t}-\hat{C}_{t}\|^2_{\infty}\\
& \leq 2 \sum_{t \in \cT}\|C_{t}-\hat{C}_{t}\|^2_{\infty}.\label{50}
\end{align}
On the one hand, by Triangle Inequality,
\begin{align}
\sum_{t \in \cT} \|(\hat{W}\hat{E})_{t}-C_{t0}\|^2_{\infty} &\leq
\sum_{t \in \cT} \;(\|(\hat{W}\hat{E})_t-\hat{C}_t\|_{\infty}+
 \|\hat{C}_t-C_t\|_{\infty}+\|C_t-C_{t0}\|_{\infty})^2\\
&\leq 3\sum_{t \in \cT} \|(\hat{W}\hat{E})_{t}-\hat{C}_t\|^2_{\infty}+
 3\sum_{t \in \cT}\|\hat{C}_t-C_t\|_{\infty}^2+3\sum_{t \in \cT}\|C_t-C_{t0}\|_{\infty}^2.
\end{align}
In addition, by \aspref{curve} and \eqref{50},
\begin{align}
\sum_{t \in \cT} \|(\hat{W}\hat{E})_{t}-C_{t0}\|^2_{\infty} &\leq 6\sum_{t \in \cT}\|C_t-\hat{C}_t\|_{\infty}^2+3\sum_{t \in \cT}\|\hat{C}_t-C_t\|_{\infty}^2+3\sum_{t \in \cH}\|C_t-C_{t0}\|_{\infty}^2\\
& = 9\sum_{t \in \cT} \|\hat{C}_t-C_{t}\|^2_{\infty}+3\sum_{t \in \cH}\|C_t-C_{t0}\|_{\infty}^2\\
& \leq 9\sum_{t \in \cT} \|\hat{C}_t-C_{t}\|^2_{\infty}+\frac{n(\log n)^\beta}{8}.
\end{align}
On the other hand,
\begin{align}
\sum_{t \in \cT} \|(\hat{W}\hat{E})_t-C_{t0}\|^2_{\infty}
\geq  \sum_{t \in \cA'} \frac{\Delta^2}{4}= \frac{\Delta^2(|\cA_0|+|\cA_1|)}{4} = \frac{|\cA'|\Delta^2}{4},
\end{align}
then we have 
\begin{align}
|\cA'|&\leq \frac{ 36\sum_{t \in \cT} \|\hat{C}_t-C_{t}\|^2_{\infty}+\frac{n(\log n)^\beta}{2}}{\Delta^2}.\label{57}
\end{align}
By \lemref{cluster}, 
\begin{equation}
\P\Big(|\cR|/n^2  >\frac{(\log n)^\beta}{\Delta^2 n}\Big) =
\P\Big(|\cR|>\frac{n(\log n)^\beta}{\Delta^2}\Big)\leq 
\P\Big(|\cA'|>\frac{n(\log n)^\beta}{\Delta^2}\Big),
\end{equation}
then by \lemref{consistency} and \eqref{57}, there exists a constant $J$,
\begin{align}
&\P\Big(|\cR|/n^2  >\frac{(\log n)^\beta}{\Delta^2 n}\Big) \\
&\leq \P\Big(\frac{36\sum_{t \in \cT} \|\hat{C}_t-C_{t}\|^2_{\infty}+\frac{n(\log n)^\beta}{2}}{\Delta^2}>\frac{n(\log n)^\beta}{\Delta^2}\Big)\\
&= \P\Big(\sum_{t \in \cT} \|\hat{C}_t-C_{t}\|^2_{\infty}>\frac{n(\log n)^\beta}{72}\Big)\\ 
&\leq \sum_{t \in \cT}\;\P\Big(\|\hat{C}_t-C_t\|^2_{\infty}>\frac{(\log n)^\beta}{72n}\Big)\\ 
&\leq 2(2\sqrt{n}+1)^2n^2 \exp(-J\cdot (\log n)^\beta).
\end{align}
So for $\forall \beta>1$, as $n \rightarrow \infty$,
\begin{equation}
\P\Big(|\cR|/n^2  >\frac{(\log n)^\beta}{\Delta^2 n}\Big)\leq 2(2\sqrt{n}+1)^2n^2 \exp(-J\cdot (\log n)^\beta) \rightarrow 0.
\end{equation}
Thus, for stationary Gaussian random field, we get the error bound of k-means clustering algorithm.
\end{proof}

\section{Non-stationary Textures}
\label{sec:non-stationary}
In this section we consider textures to be non-stationary. Here both $\cG_0$ and $\cG_1$ are non-stationary Gaussian Markov random fields with mean $0$. We add location information into consideration and cluster the $n^2$ pixels into two groups by single-linkage algorithm. The algorithm is established in \secref{nonstationary_method}. Then we show the consistency of a simple incarnation of the basic approach in \secref{nonstationary_theory}.

\subsection{Method}
\label{sec:nonstationary_method}
Pick up $\lfloor\frac{n}{2m+1}\rfloor\times \lfloor\frac{n}{2m+1}\rfloor$ pixels $\{X_u\}_{u \in \cU}$ with equal intervals from $\{X_t\}_{t\in \cT}$, where
\begin{equation}
\cU =\Big\{u= (u_1,u_2)\;\Big|\;u_1=(2m+1)\cdot t_1,\;\; u_2 =(2m+1)\cdot t_2, \;\;(t_1,t_2) = \Big\{1,2,\cdots, \Big\lfloor\frac{n}{2m+1}\Big\rfloor\Big\}^2\Big\}.
\end{equation}
Here $\cU$ is a subset of $\cT$. Similar to \eqref{patch} in \secref{stationary_method}, for $ \forall u \in \cU$, pick up patch $S_u$ as follows
\begin{equation}
S_u=
\begin{pmatrix}
   X_{u+(-m,-m)} & X_{u+(-m, -m+1)} & \cdots & X_{u+(-m,m)} \\
   X_{u+(-m+1,-m)} & X_{u+(-m+1,-m+1)} & \cdots & X_{u+(-m+1,m)} \\
   \vdots & \vdots & \cdots& \vdots \\
   X_{u+(m,-m)} & X_{u+(m,-m+1)} & \cdots& X_{u+(m,m)}
\end{pmatrix}.
\end{equation}\\
For $\forall u \neq v \in \cU$, it is obvious that $|u_1-v_1|\geq 2m+1$ or $|u_2-v_2|\geq 2m+1$. So there is no overlap between $S_u$ and $S_v$.

Next for $\forall u\in \cU$, add location information into its true feature $C_u$. Denote $C'_u$ as the new true feature of pixel $X_u$ as follows
\begin{equation}
C'_u = \Big(C_u, \frac{u_1}{n},\frac{u_2}{n}\Big),
\end{equation}
where $C'_u$ is a vector of length $(2m+1)^2+2$. Similarly, denote $\hat{C}'_u$ as the new observed feature of pixel $X_u$ as follows
\begin{equation}
\hat{C}'_u = \Big(\hat{C}_u, \frac{u_1}{n},\frac{u_2}{n}\Big),
\end{equation}
where $\hat{C}'_u$ also is a vector of length $(2m+1)^2+2$.

Apply single-linkage algorithm in following steps. Firstly among $\cU$, connect all pairs $(u,v)$ with $\|\hat{C}'_u-\hat{C}'_v\|_{\infty}<\frac{(\log n)^{\beta/2}}{\sqrt{n}}$, where $1<\beta<2$. Next for $\forall u \in \cU$, assign all the other pixels in $S_u$ into the same cluster as $X_u$. Then for any pixel $X_t$ which is still not clustered, find the pixel $X_u$ in $\{X_u\}_{u \in \cU}$ with the smallest distance to $X_t$, and assign pixel $X_t$ into the same cluster with $X_u$. Finally we obtain the clustering result.

As a summary, we provide the procedure of single-linkage algorithm in \algref{alg2}.
\begin{algorithm}[H]
  \caption{Texture Segmentation with Single-linkage Algorithm}
  \label{alg:alg2}
  \begin{algorithmic}[1]
    \REQUIRE
      Observations $\{X_t\}_{t\in \cT}$.
    \ENSURE
     Single-linkage clustering result.
    \STATE For $\forall u = (u_1,u_2) \in \cU$, pick up patch $S_u$.
    \STATE
    Calculate new observed features $\{\hat{C}'_u\}_{u\in \cU}$.
    \STATE
    Apply single-linkage algorithm based on $\{\hat{C}'_u\}_{u\in \cU}$.
    \RETURN Single-linkage clustering result.
  \end{algorithmic}
\end{algorithm}
Define the following set
\begin{equation}
\cV= \{u \in  \cU \;|\; S_u \cap  \cG_0 = \emptyset \;\text{or} \; S_u \cap  \cG_1 = \emptyset \}
\end{equation}
and
\begin{equation}
\cW = \{t \in \cT \;|\;X_t \;\text{is a pixel in patch}\; S_u \;\text{where}\; u\in \cV\}.
\end{equation}
In next section, we can show that all the pixels in $\cV$ can be clustered correctly with probability going to $1$. Thus, all pixels in $\cW$ can be clustered correctly with probability going to $1$.

\subsection{Theory}
\label{sec:nonstationary_theory}
Firstly we introduce two assumptions on the non-stationary level of the fields.
\begin{asp} 
\label{asp:D}
For any $X_t,X_s$ in the same sub-region,
\begin{equation}
\|C_t-C_s\|_{\infty}= O(\sqrt{\log n}\cdot {\sf D}(s,t)),
\end{equation}
where ${\sf D}(s,t)$ is the distance between two pixels $X_t$ and $X_s$
\begin{equation}
{\sf D}(s,t) = \sqrt{\Big(\frac{s_1-t_1}{n}\Big)^2+\Big(\frac{s_2-t_2}{n}\Big)^2}.
\end{equation}
\end{asp}

\begin{asp}
\label{asp:E}
 For any $X_t,X_s$ in different sub-regions, if ${\sf D}(s,t)< \frac{\log n}{\sqrt{n}}$, there exists a constant $K$ such that
\begin{equation}
\|C_t-C_s\|_{\infty}\geq \frac{K\cdot \log n}{\sqrt{n}}.
\end{equation}
\end{asp}
Next, we show that the single-linkage algorithm in above section is consistent.
\begin{theorem}
\label{thm:2}
Under \aspref{D} and \aspref{E}, by single-linkage clustering in \algref{alg2}, set threshold value $b=\frac{(\log n)^{\beta/2}}{\sqrt{n}}$, where $1<\beta <2$. Then as $n \rightarrow \infty$,
\begin{align} 
\P(\text{All pixels in} \;\cW\; \text{are clustered correctly}) 
&\geq 1-n^3\cdot  \exp(-J\cdot (\log n)^\beta)\;\;\rightarrow 1.
\end{align}
\end{theorem}

\begin{proof}
Set threshold value $b=\frac{(\log n)^{\beta/2}}{\sqrt{n}}$, where $1<\beta <2$. For $\forall u \in \cU$, denote $u_+$ as the pixel bordering and above $u$ in $\cU$, and denote $u_-$ as the pixel bordering and below $u$ in $\cU$, then
\begin{align}
&\P(\text{All pixels in} \;\cW\; \text{are clustered correctly}) \\
&\geq 1- \P\Big(\max_{u\in \cV}
\max_{\mbox{\tiny$\begin{array}{c}
v\in \cV\\ 
v, u\;\text{in the same sub-region} \\
S_u\; \text{borders on} \;S_v \end{array}$}} \|\hat{C}'_u-\hat{C}'_v\|_\infty >\frac{(\log n)^{\beta/2}}{\sqrt{n}}\Big)\\
&- \P\Big(\min_{u\in \cV}\min_{\mbox{\tiny$\begin{array}{c}
v\in \cV\\ 
v, u\;\text{in different sub-regions}\end{array}$}} \|\hat{C}'_u-\hat{C}'_v\|_\infty <\frac{(\log n)^{\beta/2}}{\sqrt{n}}\Big)\\
&- \P\Big(\min_{u \in \cU \backslash \cV} \|\hat{C}'_{u_+}-\hat{C}'_{u_-}\|_\infty <\frac{2(\log n)^{\beta/2}}{\sqrt{n}}\Big).
\end{align}
By Union bound,
\begin{align}
&\P(\text{All pixels in} \;\cW\; \text{are clustered correctly}) \\
&\geq 1-\sum_{u\in \cV}
\sum_{\mbox{\tiny$\begin{array}{c}
v\in \cV\\ 
v, u\;\text{in the same sub-region} \\
S_u\; \text{borders on} \;S_v \end{array}$}} \P\Big(\|\hat{C}'_u-\hat{C}'_{v}\|_\infty>\frac{(\log n)^{\beta/2}}{\sqrt{n}}\Big)\\
&-\sum_{u\in \cV}
\sum_{\mbox{\tiny$\begin{array}{c}
v\in \cV\\ 
v, u\;\text{in different sub-regions}\end{array}$}}
\P\Big(\|\hat{C}'_u-\hat{C}'_v\|_\infty <\frac{(\log n)^{\beta/2}}{\sqrt{n}}\Big)\\
&-\sum_{u \in \cU \backslash \cV}
\P\Big(\|\hat{C}'_{u_+}-\hat{C}'_{u_-}\|_\infty <\frac{2(\log n)^{\beta/2}}{\sqrt{n}}\Big).
\end{align}
We calculate above probability in three steps. Firstly, under \aspref{D}, for $\forall u \in \cU$,
\begin{equation}
\max_{\mbox{\tiny$\begin{array}{c}
v\in \cV\\ 
v, u\;\text{in the same sub-region}\\
S_u\; \text{borders on} \;S_v \end{array}$}} \|C'_u-C'_{v}\|_\infty = O\Big(\sqrt{\frac{\log n}{n}}\Big),\label{85}
\end{equation}
then by Triangle Inequality, 
\begin{align}
&\sum_{u\in \cV}\sum_{\mbox{\tiny$\begin{array}{c}
v\in \cV\\ 
v, u\;\text{in the same sub-region} \\
S_u\; \text{borders on} \;S_v \end{array}$}} \P\Big(\|\hat{C}'_u-\hat{C}'_{v}\|_\infty>\frac{(\log n)^{\beta/2}}{\sqrt{n}}\Big)\\
&\leq
\sum_{u\in \cV} \sum_{\mbox{\tiny$\begin{array}{c}
v\in \cV\\ 
v, u\;\text{in the same sub-region} \\
S_u\; \text{borders on} \;S_v \end{array}$}} \Big[ \P\Big(\|\hat{C}'_u-C'_u\|_\infty>\frac{(\log n)^{\beta/2}}{3\sqrt{n}}\Big)+\P\Big(\|\hat{C}'_{v}-C'_{v}\|_\infty>\frac{(\log n)^{\beta/2}}{3\sqrt{n}}\Big)\\
&+  \P\Big(\|C'_{u}-C'_{v}\|_\infty>\frac{(\log n)^{\beta/2}}{3\sqrt{n}}\Big)\Big]\\
&\leq
\sum_{u\in \cV} \sum_{\mbox{\tiny$\begin{array}{c}
v\in \cV\\ 
v, u\;\text{in the same sub-region} \\
S_u\; \text{borders on} \;S_v \end{array}$}} \Big[ \P\Big(\|\hat{C}'_u-C'_u\|_\infty>\frac{(\log n)^{\beta/2}}{3\sqrt{n}}\Big)+\P\Big(\|\hat{C}'_{v}-C'_{v}\|_\infty>\frac{(\log n)^{\beta/2}}{3\sqrt{n}}\Big)\Big].
\end{align}
By \lemref{consistency}, for $1<\beta <2$, there exists a constant $J$, as $n \rightarrow \infty$,

\begin{align}
&\sum_{u\in \cV}\sum_{\mbox{\tiny$\begin{array}{c}
v\in \cV\\ 
v, u\;\text{in the same sub-region} \\
S_u\; \text{borders on} \;S_v \end{array}$}} \P\Big(\|\hat{C}'_u-\hat{C}'_{v}\|_\infty>\frac{(\log n)^{\beta/2}}{\sqrt{n}}\Big)\\
&\leq |\cV| \cdot 16(2\sqrt{n}+1)^2 \exp(-J\cdot (\log n)^{\beta})\\
&\leq 16n^2 \exp(-J\cdot (\log n)^\beta) \;\; \rightarrow 0.
\end{align}\\
Secondly, for $\forall u,v \in \cU $ such that $X_u,X_v$ are in different sub-regions, if 
${\sf D}(u,v)\geq \frac{\log n}{\sqrt{n}},$
then
\begin{equation}\label{87}
\|C'_u-C'_v\|_\infty \geq \frac{\log n}{2\sqrt{n}}.
\end{equation}
If ${\sf D}(u,v)< \frac{\log n}{\sqrt{n}},$ under \aspref{E}, there exists a constant $K$, such that
\begin{equation}\label{88}
\|C'_u-C'_v\|_\infty \geq \frac{K \cdot\log n}{\sqrt{n}}.
\end{equation}
So there exists a constant $J$, as $n \rightarrow \infty$,
\begin{align}
&\sum_{u\in \cV}
\sum_{\mbox{\tiny$\begin{array}{c}
v\in \cV\\ 
v, u\;\text{in different sub-regions}\end{array}$}}
P\Big(\|\hat{C}'_u-\hat{C}'_v\|_\infty <\frac{(\log n)^{\beta/2}}{\sqrt{n}}\Big)\\
&\leq  \sum_{u\in \cV}
\sum_{\mbox{\tiny$\begin{array}{c}
v\in \cV\\ 
v, u\;\text{in different sub-regions}\end{array}$}}
\Big[P\Big(\|C'_u-C'_v\|_\infty <\frac{3(\log n)^{\beta/2}}{\sqrt{n}}\Big)+P\Big(\|\hat{C}'_u-C'_u\|_\infty >\frac{(\log n)^{\beta/2}}{\sqrt{n}}\large\Big)\\
&+P\Big(\|\hat{C}'_v-C'_v\|_\infty >\frac{(\log n)^{\beta/2}}{\sqrt{n}}\Big)\Big]\\
&\leq \frac{2n^4}{(2\sqrt{n}+1)^4}\cdot  P\Big(\|\hat{C}'_u-C'_u\|_\infty>\frac{(\log n)^{{\beta/2}}}{\sqrt{n}}\Big)\\
&\leq \frac{4n^4}{(2\sqrt{n}+1)^2}\cdot  \exp(-J\cdot (\log n)^\beta) \;\; \rightarrow 0.
\end{align}
Thirdly, for $\forall u \in \cU \backslash \cV$, $X_{u_+}$ and $X_{u_-}$ are in different sub-regions, so by \eqref{87} and \eqref{88},
\begin{align}
&\sum_{u \in \cU \backslash \cV}
\P\Big(\|\hat{C}'_{u_+}-\hat{C}'_{u_-}\|_\infty <\frac{2(\log n)^{\beta/2}}{\sqrt{n}}\Big)=0.
\end{align}
Thus, there exists a constant $J$, as $n \rightarrow \infty$,
\begin{align} 
&P(\text{All pixels in} \;\cW\; \text{are clustered correctly})\\ 
&\geq 1-16n^2 \exp(-J\cdot (\log n)^\beta)-\frac{4n^4}{(2\sqrt{n}+1)^2}\cdot  \exp(-J\cdot (\log n)^\beta)\\
&\geq 1-n^3\cdot  \exp(-J\cdot (\log n)^\beta)\;\;\rightarrow 1.
\end{align}
\end{proof}

\subsection{Example}
\subsubsection{Model}
Follow the ideas in \citep{higdon1999non}, we define non-stationary Gaussian process as follows. For any pixels $X_t$ and $X_s$, define non-stationary covariance between $X_t$ and $X_s$ to be

\begin{equation}
C(X_t,X_s) = \int_{\mathbb{R}^2}K_{t}(r)K_{s}(r)\d r,
\end{equation}
where $K_{t}(\cdot)$ and $K_{s}(\cdot)$ are Gaussian kernel functions
\begin{equation}
K_{t}(r) = \frac{1}{2\pi|\Sigma_t|^{\frac{1}{2}}}\exp{\left[-\frac{1}{2}(r-t)^T\Sigma_t^{-1}(r-t)\right]}
\end{equation}
and
\begin{equation}
K_{s}(r) = \frac{1}{2\pi|\Sigma_s|^{\frac{1}{2}}}\exp{\left[-\frac{1}{2}(r-s)^T\Sigma_s^{-1}(r-s)\right]}.
\end{equation}
It is easy to check the covariance is non-negative definite. Then we create the non-stationary process by convoluting the white noise process $\phi(\cdot)$ with kernel function $K_t(\cdot)$,
\begin{equation}
X_t = \int_{\mathbb{R}^2} K_{t}(r)\;\d \phi(r) \quad \text{for} \quad t \in \cT.
\end{equation}
Next, we simplify the covariance matrix. Suppose $B \sim N(0, \Sigma_t)$ and $D \sim N(s, \Sigma_s)$, where $B$ and $D$ are independent. Let $g_B(\cdot)$, $g_D(\cdot)$ and $g_{D-B}(\cdot)$ denote the density functions of $B, D$ and $D-B$. Similarly, $g_{B,D}(\cdot)$ and $g_{D-B,D}(\cdot)$ are join density functions. Then 
\begin{align} 
C(X_t,X_s) &= \int_{R^2}K_{t}(r)K_{s}(r)\d r
=\int_{R^2} g_B(r-t) \cdot g_D(r)\;\d r \\
&=\int_{R^2} g_{B,D}(r-t,r)\;\d r
= \int_{R^2} g_{D-B,D}(t,r) \;\d r\\
&= g_{D-B}(t) \cdot \int_R^2 g_{D}(r)\;\d r
= g_{D-B}(t).
\end{align}
Since $D-B \sim N(s, \Sigma_t+\Sigma_s)$, we have 
\begin{equation}
C(X_t, X_s) = g_{D-B}(t) = \frac{1}{2\pi|\Sigma_t+\Sigma_s|^{\frac{1}{2}}}\exp{\left[-\frac{1}{2}(t-s)^T(\Sigma_t+\Sigma_s)^{-1}(t-s)\right]}.
\end{equation}
\subsubsection{Assumptions and Theory on Example Model}

Each pixel $X_t$ has it own kernel function $K_{t}(\cdot)$, and the non-stationary process is controlled by kernel functions $\{K_{t}(\cdot)\}_{t\in \cT}$. For each pixel $X_t$, it has Gaussian kernel function 
\begin{equation}
K_{t}(r) = \frac{1}{2\pi|\Sigma_t|^{\frac{1}{2}}}\exp{\left[-\frac{1}{2}(r-t)^T\Sigma_t^{-1}(r-t)\right]}.
\end{equation}
The only parameters of the non-stationary process are the covariance matrices $\{\Sigma_t\}_{t \in \cT}$ of the kernel functions. Here we call $\{\Sigma_t\}_{t \in \cT}$ as the “size" of the kernel functions. If all pixels $X_t$ have the same "size" $\Sigma_t$, the field is stationary. For each pixel $X_t$, $t=(t_1,t_2)$ is a 2-dimension vector, so its "size" $\Sigma_t$ is a $2\times 2$ matrix. Denote it as
\begin{equation}
\Sigma_t=
 \begin{pmatrix}
   a_t & b_t \\
   c_t & d_t
  \end{pmatrix}.
\end{equation}
For any pixels $X_t$ and $X_s$, define
\begin{equation}
{\sf d}(s,t) = \max\{|a_t-a_s|,|b_t-b_s|,|c_t-c_s|,|d_t-d_s|\}.
\end{equation}
Next, similar to \aspref{D} and \aspref{E}, for the example kernel convolution model, we introduce assumptions directly on the "size" of kernel function.

\begin{asp} 
\label{asp:D'}
For $\forall X_t,X_s$ in the same sub-region, 
\begin{equation}
{\sf d}(s,t)= O(\sqrt{\log n}\cdot {\sf D}(s,t)).
\end{equation}
\end{asp}

\begin{asp}
\label{asp:E'}
For $\forall X_t,X_s$ in different sub-regions, if ${\sf D}(s,t)< \frac{\log n}{\sqrt{n}}$, there exists a constant $K$ such that
\begin{equation}
{\sf d}(s,t)\geq \frac{K\cdot \log n}{\sqrt{n}}.
\end{equation}
\end{asp}

From \lemref{3}, all conditions in \thmref{2} are satisfied, so \thmref{2} works here. Thus, under \aspref{D'} and \aspref{E'}, single-linkage algorithm is consistent under the example kernel model.

\section{Numerical Experiments}
\label{sec:numerics}

In our experiments, for the sake of stability, we use a form of size-constrained k-means \citep{wagstaff2001constrained, bradley2000constrained} , size-constrained single-linkage and size-constrained ward-linkage algorithms.

\subsection{Synthetic Stationary Textures}

In the section, we create several stationary random field models by moving average on Gaussian noise. Suppose white noise $Z_t \sim N(0,1)$. We generate four stationary random fields as follows
\begin{align}
&\text{Model 1: } \;\;X_t= \sum_{i=-m}^{m} Z_{t+(i,i)}. \\
& \text{Model 2: } \;\;X_t= \sum_{i=-m}^{m} Z_{t+(-i,i)}. \\
& \text{Model 3: } \;\; X_t= \sum_{i=-m}^{m} Z_{t+(0,i)}.\\
& \text{Model 4: }\;\; X_t= \sum_{i=-m}^{m} Z_{t+(i,0)}. 
\end{align}

Based on above four models, after standardization and combination, we obtain six $128\times 128$ mosaics, which are showed in Figure \ref{Fig.main.1}. Each mosaic contains two different textures, and it is divided by a straight line in the middle.

\begin{figure}[htbp]
\centering 
\subfigure[Model 1 vs Model 2]{
\label{Fig.sub.11}
\includegraphics[width=0.25\textwidth]{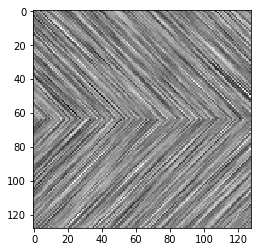}}
\subfigure[Model 1 vs Model 3]{
\label{Fig.sub.12}
\includegraphics[width=0.25\textwidth]{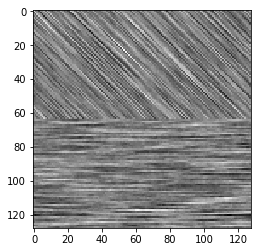}}
\subfigure[Model 1 vs Model 4]{
\label{Fig.sub.13}
\includegraphics[width=0.25\textwidth]{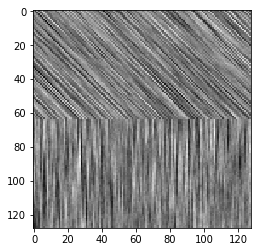}}
\subfigure[Model 2 vs Model 3]{
\label{Fig.sub.14}
\includegraphics[width=0.25\textwidth]{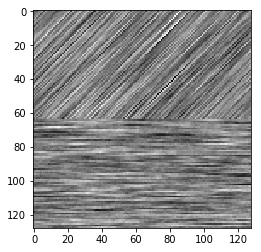}}
\subfigure[Model 2 vs Model 4]{
\label{Fig.sub.15}
\includegraphics[width=0.25\textwidth]{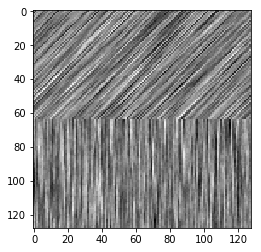}}
\subfigure[Model 3 vs Model 4]{
\label{Fig.sub.16}
\includegraphics[width=0.25\textwidth]{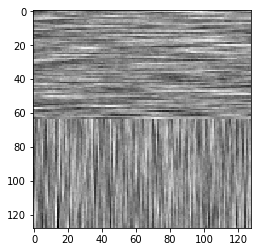}}
\caption{GMRF texture mosaics}
\label{Fig.main.1}
\end{figure}
We apply $\{\hat{C}'_t\}_{t\in \cT}$ as the observed features of pixels $\{X_t\}_{T \in \cT}$. Consider three clustering algorithms: size-constrained single-linkage, size-constrained ward-linkage and size-constrained k-means algorithms. We present segmentation accuracy in Table~\ref{table1}. Single-linkage works but does not perform well, and ward-linkage improves it. K-means algorithm works best and segmentation accuracy is almost 1. 

\begin{table}[htbp]
\centering
\begin{tabular}{c cc c} 
\toprule
Mosaic  & Single-linkage & Ward-linkage & K-means \\
\midrule
Model 1 vs Model 2 & 0.7362 & 0.9665 & 0.9868 \\ 
\midrule
Model 1 vs Model 3 & 0.9144 & 0.9632 & 0.9820\\
\midrule
Model 1 vs Model 4 & 0.8329 & 0.9677  & 0.9867\\
\midrule
Model 2 vs Model 3 & 0.9210 &0.9657 & 0.9822 \\
\midrule
Model 2 vs Model 4 & 0.8190 &0.9661 & 0.9868 \\ 
\midrule
Model 3 vs Model 4 & 0.7910 & 0.9626 & 0.9816\\ 
\midrule
Mean Value & 0.8358 & 0.9653 & 0.9844\\ 
\bottomrule
\end{tabular}
\caption{Segmentation accuracy}
\label{table1}
\end{table}

\subsection{Natural Textures}
In this section, We pick up textures from Brodatz  album \citep{brodatz1966textures}.

\subsubsection{Two Regions Divided by a Straight Line}
\label{sec:Straight}
We pick up three textures from Brodatz album \citep{brodatz1966textures}: $D21$, $D55$ and $D77$. After standardization and combination, we obtain three $160\times 160$ mosaics, which are showed in \figref{Fig.main.2}. Each mosaic contains two different Brodatz textures, and they are divided by a straight line in the middle.
\begin{figure}[htbp]
\centering 
\subfigure[D21 vs D55]{
\label{Fig.sub.21}
\includegraphics[width=0.25\textwidth]{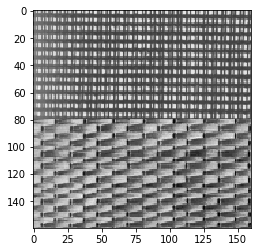}}
\subfigure[D21 vs D77]{
\label{Fig.sub.22}
\includegraphics[width=0.25\textwidth]{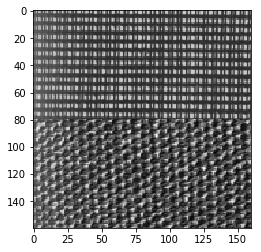}}
\subfigure[D55 vs D77]{
\label{Fig.sub.23}
\includegraphics[width=0.25\textwidth]{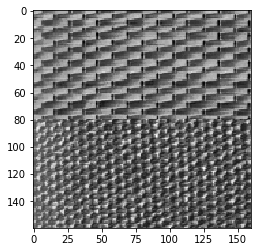}}
\caption{Texture mosaics}
\label{fig:Fig.main.2}
\end{figure}

We apply $\{\hat{C}'_t\}_{t\in \cT}$ as the observed features of pixels $\{X_t\}_{T \in \cT}$. Consider three clustering algorithms: size-constrained single-linkage, size-constrained ward-linkage and size-constrained k-means algorithms. Segmentation results are shown in \figref{3}. Also we present segmentation accuracy in Table~\ref{table2}. All three algorithms work perfectly and segmentation accuracy is almost 1. 

\begin{figure}[htbp]
\centering 
\subfigure[D21 vs D55]{
\label{fig:31}
\includegraphics[width=0.85\textwidth]{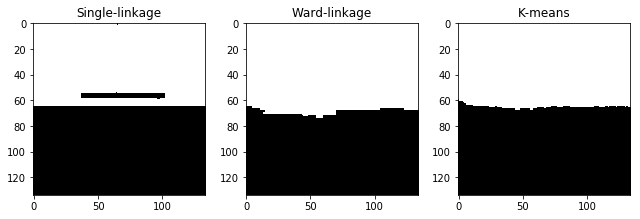}}
\subfigure[D21 vs D77]{
\label{fig:32}
\includegraphics[width=0.85\textwidth]{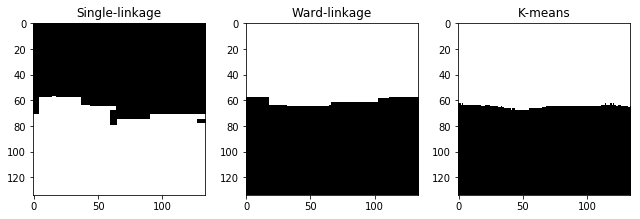}}
\subfigure[D55 vs D77]{
\label{fig:33}
\includegraphics[width=0.85\textwidth]{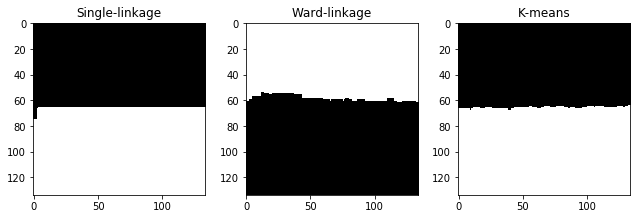}}
\caption{Segmentation results}
\label{fig:3}
\end{figure}

\begin{table}[htbp]
\centering
\begin{tabular}{c cc c} 
\toprule  
Mosaics & Single-linkage & Ward-linkage & K-means  \\
\midrule
 D21 vs D55  & 0.9703 & 0.9819 & 0.9891\\
\midrule
 D21 vs D77 & 0.9536 & 0.9592 & 0.9858\\
\midrule
D55 vs D77 & 0.9914 & 0.9396 & 0.9928\\
\midrule
 Mean Value & 0.9718 & 0.9602 & 0.9892\\
\bottomrule
\end{tabular}
\caption{Segmentation accuracy}
\label{table2}
\end{table}

\subsubsection{Two Regions Divided by a Curve}
Same to \secref{Straight}, we still run simulations on textures $D21$, $D55$ and $D77$. Here in each mosaic, the textures are divided by a circle in the middle, as shown in \figref{Fig.main.4}.
\begin{figure}[htbp]
\centering 
\subfigure[D21 vs D55]{
\label{Fig.sub.41}
\includegraphics[width=0.25\textwidth]{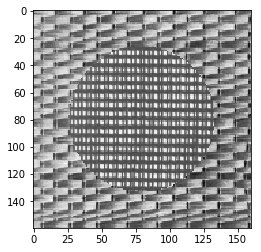}}
\subfigure[D21 vs D77]{
\label{Fig.sub.42}
\includegraphics[width=0.25\textwidth]{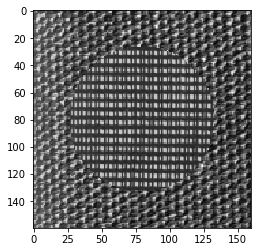}}
\subfigure[D55 vs D77]{
\label{Fig.sub.43}
\includegraphics[width=0.25\textwidth]{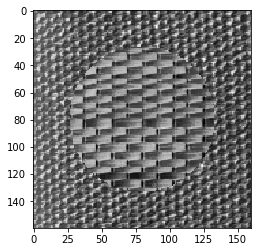}}
\caption{Texture mosaics}
\label{fig:Fig.main.4}
\end{figure}

We apply $\{\hat{C}'_t\}_{t\in \cT}$ as the observed features of pixels $\{X_t\}_{T \in \cT}$. Consider three clustering algorithms: size-constrained single-linkage, size-constrained ward-linkage and size-constrained k-means algorithms. Segmentation results are shown in \figref{5}. Also we present segmentation accuracy in Table \ref{table3}. Single-linkage works but does not perform well, and ward-linkage improves it. K-means algorithm works perfectly and segmentation accuracy is almost 1. 
\begin{figure}[htbp]
\centering 
\subfigure[D21 vs D55]{
\label{fig:51}
\includegraphics[width=0.85\textwidth]{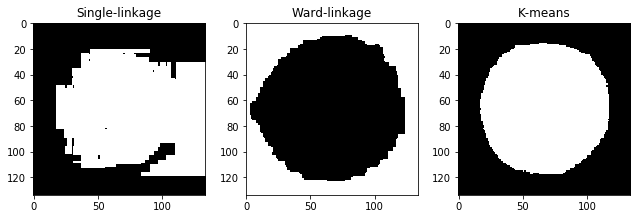}}
\subfigure[D21 vs D77]{
\label{fig:52}
\includegraphics[width=0.85\textwidth]{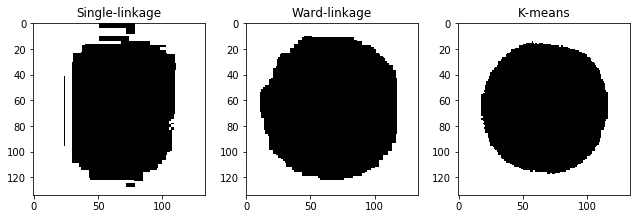}}
\subfigure[D55 vs D77]{
\label{fig:53}
\includegraphics[width=0.85\textwidth]{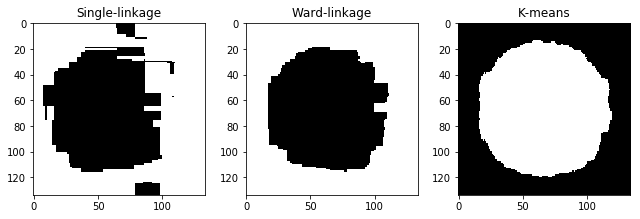}}
\caption{Segmentation results}
\label{fig:5}
\end{figure}

\begin{table}[htbp]
\centering
\begin{tabular}{c cc c} 
\toprule  
Mosaic & Single-linkage & Ward-linkage & K-means \\
\midrule
 D21 vs D55 & 0.8093 & 0.9332 & 0.9562\\
\midrule  
 D21 vs D77 & 0.8903 & 0.9563 & 0.9418\\
\midrule  
D55 vs D77 & 0.8295 & 0.8992 & 0.9739\\
\midrule  
 Mean Value & 0.8430 & 0.9296 & 0.9573\\
\bottomrule
\end{tabular}
\caption{Segmentation accuracy}
\label{table3}
\end{table}

\subsubsection{Multiple Regions}
\label{sec:multiple_region}
Here we construct three mosaics from eight textures in Brodatz album \citep{brodatz1966textures}: $D4$, $D6$, $D20$, $D21$, $D34$, $D52$, $D55$ and $D77$. After standardization and combination, we obtain three $160\times 160$ mosaics, which are showed in  \figref{Fig.main.6}. Each mosaic contains four different Brodatz textures, and they are divided by horizontal and vertical lines in the middle.
\begin{figure}[htbp]
\centering 
\subfigure[D4, D6, D20 \& D52]{
\label{Fig.sub.61}
\includegraphics[width=0.25\textwidth]{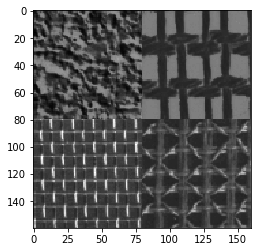}}
\subfigure[D21,D34, D55 \& D77]{
\label{Fig.sub.62}
\includegraphics[width=0.25\textwidth]{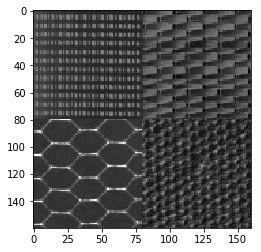}}
\subfigure[D6, D21, D34 \& D77]{
\label{Fig.sub.63}
\includegraphics[width=0.25\textwidth]{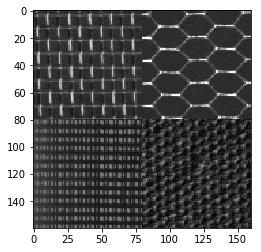}}
\caption{Texture mosaics}
\label{fig:Fig.main.6}
\end{figure}

We apply $\{\hat{C}'_t\}_{t\in \cT}$ as the observed features of pixels $\{X_t\}_{T \in \cT}$. Consider three clustering algorithms: size-constrained single-linkage, size-constrained ward-linkage and size-constrained k-means algorithms. Segmentation results are shown in \figref{7}. Also we present segmentation accuracy in Table~\ref{table4}. For multi-cluster mosaics, single-linkage works but does not perform well, and ward-linkage improves it. K-means algorithm works perfectly and segmentation accuracy is almost 1. 

\begin{figure}[htbp]
\centering 
\subfigure[D04, D06, D20 \& D52]{
\label{fig:71}
\includegraphics[width=0.85\textwidth]{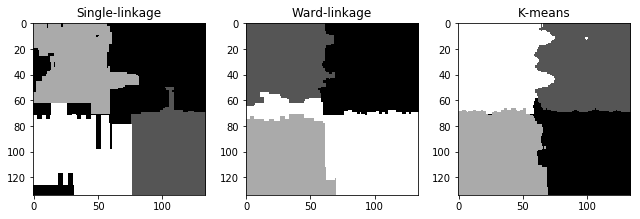}}
\subfigure[D21, D34, D55 \& D77]{
\label{fig:72}
\includegraphics[width=0.85\textwidth]{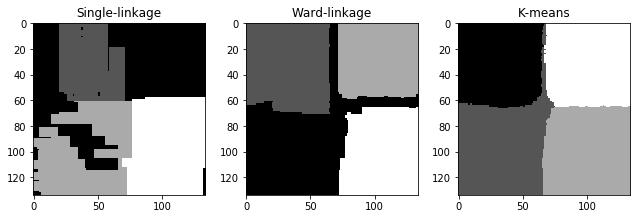}}
\subfigure[D06, D21, D34 \& D77]{
\label{fig:73}
\includegraphics[width=0.85\textwidth]{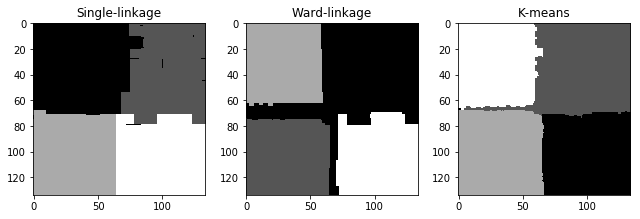}}
\caption{Segmentation results}
\label{fig:7}
\end{figure}

\begin{table}[htbp]
\centering
\begin{tabular}{c c c c} 
\toprule  
  Mosaic & Single-linkage & Ward-linkage & K-means \\
\midrule
 D04, D06, D20 \& D52 & 0.8478 & 0.8985 & 0.9507\\
\midrule
 D21, D34, D55 \& D77 & 0.7530 & 0.8969 & 0.9663\\
\midrule
D06, D21, D34 \& D77 & 0.9158 & 0.8767 & 0.9465\\
\midrule
 Mean Value & 0.8389 & 0.8907 & 0.9545\\
\bottomrule
\end{tabular}
\caption{Segmentation accuracy}
\label{table4}
\end{table}

\section{Discussion}
For non-stationary textures, instead of single linkage clustering, we could also use DBSCAN algorithm. When MinPts $= 2$, DBSCAN is very similar to single-linkage algorithm, but DBSCAN includes a step for removing noisy observations. In this paper, these noisy observations could be the patches that overlap with the edge between the two sub-regions. In practice, in different settings and backgrounds, we can apply different MinPts values in DBSCAN algorithm.

Also in \secref{stationary} and \secref{non-stationary}, we only theoretically show the cases where there are only two sub-regions. Actually when there are more than two regions, the introduced algorithms also work. For example, we indicate the scene with four sub-regions in \secref{multiple_region}.

In this paper, we do clustering based on local second moment information on patches. Actually beyond sample autocovariance, we could also use higher-order statistics or other features such as SIFT \citep{lowe1999object} for more complex textures that are not necessarily Gaussian, and speculate that similar results could also be obtained in these situations.
 
\subsection*{Acknowledgments}
I am grateful to my PhD advisor, Ery Arias-Castro, for suggesting this topic and for his assistance and support throughout this project. I also want to thank Danna Zhang for helpful discussions regarding Gaussian MRFs.

\bibliographystyle{chicago}
\bibliography{ref}

\appendix
\section{Miscellanea}

\subsection{Auxiliary results}

\begin{lem}
\label{lem:3}
For any pixels $X_t, X_s$ in the image, 
\begin{equation}
\|C_t-C_s\|_\infty = O({\sf d}(s,t)).
\end{equation}
\end{lem}

\begin{proof}
For any pixel $X_t$ and $\forall i=(i_1,i_2) \in \cM$,
\begin{align}
C_t(i) &= \text{Mean value of }\{ C(X_t, X_{t+i})\;|\;\; \text{both} \;\;X_{t}\; \;\text{and} \;\;X_{t+i}\;\; \text{are in} \;S_t \;\}.
\end{align} 
Also we have
\begin{align} 
& C(X_{t}, X_{t+i})\\ 
&= \frac{1}{(2\pi)^2|\Sigma_t+\Sigma_{t+i}|^{\frac{1}{2}}}\cdot\exp{\left(-\frac{1}{2}(\frac{i}{n})(\Sigma_t+\Sigma_{t+i})^{-1}(\frac{i}{n})^T\right)}\\
&=
\frac{1}{(2\pi)^2\left|\begin{pmatrix}
   a_t+a_{t+i} & b_t+b_{t+i} \\
   c_t+c_{t+i} & d_t+d_{t+i}
  \end{pmatrix}\right|^{\frac{1}{2}}}\cdot
\exp{\left(-\frac{1}{2}(\frac{i}{n})
\begin{pmatrix}
   a_t+a_{t+i} & b_t+b_{t+i} \\
   c_t+c_{t+i} & d_t+d_{t+i}
  \end{pmatrix}^{-1}(\frac{i}{n})^T\right)}\\
&=
\frac{1}{(2\pi)^2\sqrt{
   (a_t+a_{t+i})(d_t+d_{t+i})-(b_t+b_{t+i})(c_t+c_{t+i})}}\cdot\\
   &
\exp{\left(-\frac{1}{2(a_t+a_{t+i})(d_t+d_{t+i})-(b_t+b_{t+i})(c_t+c_{t+i})}(\frac{i}{n})
\begin{pmatrix}
   d_t+d_{t+i} & -b_t-b_{t+i} \\
   -c_t-c_{t+i} & a_t+a_{t+i}
  \end{pmatrix}(\frac{i}{n})^T\right)}.\\
\end{align}
For $\forall t\in \cT$ and $i\in \cM$, let
\begin{equation}
I_{t,i} = (a_t+a_{t+i})(d_t+d_{t+i})-(b_t+b_{t+i})(c_t+c_{t+i}),
\end{equation}
then
\begin{align} 
& C(X_{t}, X_{t+i})\\ 
&= 
\frac{1}{(2\pi)^2\sqrt{I_{t,i}}}\cdot
\exp{\left(-\frac{1}{2I_{t,i}}(\frac{i}{n})
\begin{pmatrix}
   d_t+d_{t+i} & -b_t-b_{t+i} \\
   -c_t-c_{t+i} & a_t+a_{t+i}
  \end{pmatrix}(\frac{i}{n})^T\right)}.
\end{align}
Also for $\forall t\in \cT$ and $i\in \cM$, let 
\begin{equation}
R_{t,i} = \frac{1}{(2\pi)^2\sqrt{I_{t,i}}}
\end{equation}
and 
\begin{equation}
Q_{t,i} = \exp{\left(-\frac{1}{2I_{t,i}}(\frac{i}{n})
\begin{pmatrix}
   d_t+d_{t+i} & -b_t-b_{t+i} \\
   -c_t-c_{t+i} & a_t+a_{t+i}
  \end{pmatrix}(\frac{i}{n})^T\right)},
\end{equation}
then
\begin{equation}
C(X_{t}, X_{t+i}) = R_{t,i}\cdot Q_{t,i}.
\end{equation}
Similarly, for any pixel $X_s \neq X_t$, 
\begin{equation}
C(X_{s}, X_{s+i}) = R_{s,i}\cdot Q_{s,i}.
\end{equation}
Next, we work on the bound of $|C(X_{t}, X_{t+i})-C(X_{s}, X_{s+i})|$. Since
\begin{equation}
|R_{t,i}-R_{s,i}| = \left|\frac{1}{(2\pi)^2\sqrt{I_{t,i}}}-\frac{1}{(2\pi)^2\sqrt{I_{s,i}}}\right|=O({\sf d}(s,t))
\end{equation}
and
\begin{equation}
|Q_{t,i}-Q_{s,i}| = Q_{t,i}\cdot O\Big(\frac{{\sf d}(s,t)}{n}\Big)= O\Big(\frac{{\sf d}(s,t)}{n}\Big), 
\end{equation}
we have 
\begin{align}
|C(X_{t}, X_{t+i})-C(X_{s}, X_{s+i})| &= |R_{t,i}\cdot Q_{t,i}-R_{s,i}\cdot Q_{s,i}|\\
& \leq  R_{t,i}\cdot |Q_{t,i}-Q_{s,i}| + Q_{s,i}\cdot
|R_{t,i}-R_{s,i}|\\
&\leq O\Big(\frac{{\sf d}(s,t)}{n}\Big)+ O({\sf d}(s,t))\\
&=  O({\sf d}(s,t)).
\end{align}
Then for any pixels $X_t, X_s$ in the image, for $\forall i \in \cM$,
\begin{equation}
|C_t(i)-C_s(i)|= O({\sf d}(s,t)).
\end{equation}
Thus,
\begin{equation}
\|C_t-C_s\|_\infty = O({\sf d}(s,t)).
\end{equation}
\end{proof}

\end{document}